\documentclass[10pt, article]{amsart}

\usepackage{tikz}
\usetikzlibrary{calc}
\usepackage{ae} % or {zefonts}
\usepackage[T1]{fontenc}
\usepackage[cp1250]{inputenc}
\usepackage{amsmath}
\usepackage{amssymb, amsfonts,amscd,verbatim}

\usepackage[normalem]{ulem}
\usepackage{hyperref}
\usepackage{indentfirst}
\usepackage{latexsym}
\input pstricks.tex
\input xy
\xyoption{all}

\usepackage{amsmath}    % need for subequations
\theoremstyle{plain}
\newtheorem{Pocz}{Poczatek}[section]
\newtheorem{Proposition}[Pocz]{Proposition}
\newtheorem{Theorem}[Pocz]{Theorem}
\newtheorem{Corollary}[Pocz]{Corollary}

\newtheorem{Lemma}[Pocz]{Lemma}
\newtheorem{Observation}[Pocz]{Observation}

\newtheorem{Problem}[Pocz]{Problem}

\theoremstyle{definition}
\newtheorem{Definition}[Pocz]{Definition}

\theoremstyle{remark}
\newtheorem{Remark}[Pocz]{Remark}
\newtheorem{Exercise}[Pocz]{Exercise}

\DeclareMathOperator*{\Leb}{Leb}
\DeclareMathOperator*{\diam}{diam}
\DeclareMathOperator*{\st}{st}
\DeclareMathOperator*{\hor}{hor}
\DeclareMathOperator*{\halo}{halo}

\numberwithin{equation}{section}
\title[Coarse amenability versus paracompactness]
{Coarse amenability versus paracompactness}

\author{M.~Cencelj}
\address{IMFM,
Univerza v Ljubljani,
Jadranska ulica 19,
SI-1111 Ljubljana,
Slovenija }
\email{matija.cencelj@guest.arnes.si}

\author{J.~Dydak}
\address{University of Tennessee, Knoxville, TN 37996, USA}
\email{jdydak@utk.edu}

\author{A.~Vavpeti\v c}
\address{Fakulteta za Matematiko in Fiziko,
Univerza v Ljubljani,
Jadranska ulica 19,
SI-1111 Ljubljana,
Slovenija }
\email{ales.vavpetic@fmf.uni-lj.si}

\date{ \today
}
\keywords{amenability, asymptotic dimension, coarse geometry, Lipschitz maps, paracompactness, Property A}

\subjclass[2000]{Primary 54F45; Secondary 55M10}

\thanks{This research was supported by the Slovenian Research
Agency grants P1-0292-0101 and J1-2057-0101.}

\begin{document}
\maketitle
\begin{center}
\today
\end{center}

\begin{abstract}

Recent research in coarse geometry revealed similarities between certain concepts of analysis, large scale geometry, and topology. Property A of G.Yu is the coarse analog of amenability for groups and its generalization (exact spaces) was later strengthened to be the large scale analog of paracompact spaces using partitions of unity.
In this paper we go deeper into divulging analogies between coarse amenability and paracompactness. In particular, we define a new coarse analog of paracompactness modeled on the defining characteristics of expanders. That analog gives an easy proof of three categories of spaces being coarsely non-amenable: expander sequences, graph spaces with girth approaching infinity, and unions of powers of a finite non-trivial group.

\end{abstract}

\section{Introduction}

This paper is about unification, at the large scale level, of two concepts that seemed utterly different until recently: amenability in analysis and paracompactness in topology.
Amenability, only applicable in the case of countable groups, has many very diverse but equivalent definitions. The same is true of large scale analogs of amenability in case of metric spaces of bounded geometry. Our goal is to go beyond the scope of spaces of bounded geometry and we propose a new definition of coarsely amenable spaces that is stronger than all previously known analogs of amenability yet is equivalent to them for spaces of bounded geometry.

Historically speaking, the first coarse analog of amenability was proposed by G.Yu \cite{Yu00} under the name of Property A in order to solve the Novikov Conjecture. Dadarlat and Guentner \cite{DaGu} generalized Property A to the class of exact spaces. Cencelj, Dydak, and Vavpeti\v c \cite{CenDyVav1} realized that both definitions aim at dualization of paracompactness, so they
defined large scale paracompact spaces by strengthening the definition of Dadarlat-Guentner and justified it by creating analogs of theorems characterizing spaces of covering dimension at most $n$ via pushing maps into $n$-skeleta of simplicial complexes.

Here is a diagram illustrating the relationship between various analogs of amenability (with arrows representing implications):
$$
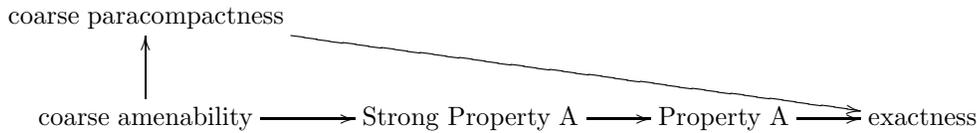
\xymatrix{ \mbox{coarse paracompactness}  \ar[drrr]&&&\\
 \mbox{coarse amenability}\ar[u]\ar[r]& \mbox{Strong Property A} \ar[r]& \mbox{Property A} \ar[r]&  \mbox{exactness}\\
}
$$
In case of spaces of bounded geometry all those properties are equivalent. We do not know of any examples distinguishing the above concepts in the class of general metric spaces.

\subsection{Organization of the paper} In Section \ref{BasicConceptsSection} we formulate Rules of Dualization for translating concepts from topology to large scale geometry. One part of Rules deals with coverings and the other part deals with partitions of unity. When trying to dualize paracompactness one has the choice of trying to use the definition of paracompactness via covers or the definition of paracompactness via partitions of unity.
It turns out applying the Rules to the cover definition of paracompactness leads to the concept of large scale weakly paracompact spaces which is different from the concept of large scale paracompact spaces introduced via dualization of the definition of paracompactness using partitions of unity.

In Section \ref{BarycentricPUs} we use partitions of unity to explain amenability of groups and Property A of G.Yu.

In Section \ref{LSParacompactness} we investigate properties of large scale weakly paracompact spaces and large scale paracompact spaces.

In Section \ref{StrongPropertyASection} we follow the Rules of Dualization and introduce the concept of Strong Property A. The main result of this section is that, in the class of large scale finitistic spaces, Strong Property A is equivalent to both Property A and large scale paracompactness.

In Section \ref{CoarseAmenabilitySection} we investigate properties of coarsely amenable spaces. The main result of this section is that, in the class of coarsely doubling spaces (that class contains all spaces of bounded geometry), coarse amenability is equivalent to both Property A and large scale paracompactness.

In Section \ref{ExpandersSection} we introduce the concept of an expander light sequence, we show expander light sequences are not coarsely amenable, and we use it to demonstrate that three major classes are not coarsely amenable: expanders, graph sequences with girth approaching infinity, and spaces constructed by P.Nowak \cite{nowak}.

In Section \ref{AddendumSection} we show that coarsely amenable spaces have MSP (metric sparsification property) thus providing an argument that coarse amenability is the best generalization of Property A for all metric spaces.

We are grateful to Misha Levin for a slick proof of \ref{SeparableAreLSWeakPara}.

We are extremely grateful to the referee for several comments and suggestions that greatly improved the exposition of the paper.

\section{Basic concepts}\label{BasicConceptsSection}

In this sections we go over some basic concepts in topology (which we think of as part of mathematics at the small scale) and we dualize them to the large scale. There are two basic ways to introduce concepts in topology: using open covers, and using partitions of unity.

\begin{Remark}\label{RulesOfDualizationRemark}
We adhere to the following

\textbf{Rules of dualization}:\\
1. Open covers are replaced by uniformly bounded covers.\\
2. Refining covers is replaced by coarsening covers.\\
3. Continuity of partitions of unity is replaced by $(\epsilon,R)$-continuity or by being $(\epsilon,\epsilon)$-Lipschitz.
\end{Remark}

\subsection{Approach via covers}

First of all, let us start from a useful class in topology, a concept less common outside of general topology; the class of weakly paracompact spaces.
\begin{Definition}\label{WeakParaDef}
A topological space $X$ is \textbf{weakly paracompact} if for each open cover $\mathcal{U}$ of $X$ there is a point-finite open cover $\mathcal{V}$ of $X$ (that means each $x\in X$ belongs to only finitely many elements of $\mathcal{V}$) that refines $\mathcal{U}$.
\end{Definition}
In what follows, the space $X$ is a metric space except where explicitly stated otherwise (e.g. in definitions from general topology). Since $r$-balls $B(x,r)$ play the role of points at scale $r$, Definition \ref{WeakParaDef} is easily dualizable. First, we need a few concepts.

\begin{Definition}\label{DiamDef}
The \textbf{diameter} $\diam(\mathcal{U})$ of a family of subsets of %a metric space
$X$ is the infimum of all $\infty\ge r \ge 0$ such that $d_X(x,y) < r$ whenever $x,y$ belong to the same element of $\mathcal{U}$. If $\diam(\mathcal{U}) <\infty$, we say $\mathcal{U}$ is \textbf{uniformly bounded}.
\end{Definition}

\begin{Definition}\label{LebesgueDef}
The \textbf{Lebesgue number} $\Leb(\mathcal{U})$ of a cover of %a metric space
$X$ is the supremum of all $r \ge 0$ such that every $r$-ball $B(x,r)$
is contained in some element of $\mathcal{U}$.
\end{Definition}

A cover at scale $r$ should have Lebesgue number at least $r$, so here is a dualization of Definition \ref{WeakParaDef} according to the Rules Of Dualization \ref{RulesOfDualizationRemark} (see \ref{CharacterisationLSWP} for other, equivalent ways, of dualizing \ref{WeakParaDef}):

\begin{Definition}\label{LSWeakParaDef}
$X$ is \textbf{large scale weakly paracompact} if for all $r, s > 0$ there is a uniformly bounded cover $\mathcal{U}$ of $X$ of Lebesgue number at least $s$ such that every $r$-ball $B(x,r)$ is contained in only finitely many elements of $\mathcal{U}$.
\end{Definition}

Recall the original definition of paracompactness by Dieudonn\'e \cite{Dieu}:
\begin{Definition}\label{ParaDef}
A topological space $X$ is \textbf{paracompact} if for each open cover $\mathcal{U}$ of $X$ there is a locally finite open cover $\mathcal{V}$ of $X$ (that means each $x\in X$ has a neighborhood $W_x$ that intersects only finitely many elements of $\mathcal{V}$) that refines $\mathcal{U}$.
\end{Definition}

To dualize \ref{ParaDef} let us express the meanings of \ref{WeakParaDef} and \ref{ParaDef}
in terms of scales: 0-scale is at the level of points and a positive scale is at the level of open covers of $X$ (with refining corresponding to decreasing of the scale, and coarsening corresponding to increasing of the scale). Thus $0 < \mathcal{U}$ means that interiors of
elements of $\mathcal{U}$ cover $X$, and $\mathcal{U}\leq \mathcal{V}$ means that $\mathcal{U}$ is a refinement of $\mathcal{V}$.
\begin{Definition}\label{HorizonDef}
Given a cover $\mathcal{U}$ of $X$ and $A\subset X$, by the \textbf{horizon} $\hor(A,\mathcal{U})$ of $A$ at scale $\mathcal{U}$ we mean
$\{U\in \mathcal{U} | A\cap U\ne\emptyset\}$.
\end{Definition}

\begin{Observation}\label{HorWeakParaObs}
A topological space $X$ is \textbf{weakly paracompact} if for each positive scale $\mathcal{U}$ of $X$ there is positive scale $\mathcal{V}\leq \mathcal{U}$ of $X$ such that each horizon $\hor(x,\mathcal{V})$, $x\in X$, is finite.
\end{Observation}

\begin{Observation}\label{HorParaObs}
A topological space $X$ is \textbf{paracompact} if for each positive scale $\mathcal{U}$ of $X$ there are positive scales $\mathcal{V}\leq \mathcal{W}\leq \mathcal{U}$ such that the horizon $\hor(V,\mathcal{W})$ of each $V\in \mathcal{V}$ is finite.
\end{Observation}

\begin{Observation}\label{HorCompObs}
A topological space $X$ is \textbf{compact} if for each positive scale $\mathcal{U}$ of $X$ there is a positive scale $\mathcal{V}\leq \mathcal{U}$ such that the horizon $\hor(X,\mathcal{V})$ of the whole $X$ is finite.
\end{Observation}

\begin{Remark}
Weak paracompactness was first defined by Arens and Dugundji in 1950 \cite{ArensDug} as metacompactness and by Bing \cite {Bing} in 1951 as pointwise paracompactness.
Observations \ref{HorWeakParaObs}, \ref{HorParaObs}, and \ref{HorCompObs} explain the original terminology.
\end{Remark}

Below is a dualization of \ref{ParaDef}. It is not totally clear that Rules \ref{RulesOfDualizationRemark} were followed here but we arrived at this definition after analyzing expanders.

\begin{Definition}\label{StrongPropADef}
%A metric space
$X$ is \textbf{coarsely amenable} if
for each $s > r > 0$ and each $\epsilon > 0$ there is a uniformly bounded cover
$\mathcal{U}$ of $X$ such that for each $x\in X$ the horizon $\hor(B(x,s),\mathcal{U})$ is finite and
$$\frac{|\hor(B(x,r),\mathcal{U})|}{|\hor( B(x,s),\mathcal{U})|} > 1- \epsilon.$$
\end{Definition}
In other words, given $x\in X$, the conditional probability of $B(x,r)\cap U\ne\emptyset$ given $B(x,s)\cap U\ne\emptyset$ for some $U\in \mathcal{U}$
can be as close to $1$ as we want.
We will see later that coarse amenability ought to be viewed as a metric analog of non-expanders.

The reason we use in \ref{StrongPropADef} the name of coarse amenability is because it implies all other large scale analogs of amenability known up to now (Property A, exactness, large scale paracompactness, Metric Sparsification Property) and is equivalent to those analogs in the class of metric spaces of bounded geometry.
\subsection{Approach via partitions of unity}

There is another way to define paracompactness (see \cite{Dyd}):
\begin{Theorem}\label{MainThmOnPara}
A topological space $X$ is paracompact if and only if for each open cover $\mathcal{U}$
there is a continuous partition of unity whose carriers refine $\mathcal{U}$.
\end{Theorem}

Traditionally (see \cite{Engelking}, for example) \ref{MainThmOnPara} is expressed
as follows: A topological space $X$ is paracompact if and only if for each open cover $\mathcal{U}$
there is a partition of unity $\{f_s\}_{s\in S}$ subordinated to $\mathcal{U}$.

What it means is that $f_s:X\to [0,1]$ for each $s\in S$, $\sum\limits_{s\in S} f_s(x)=1$
for each $x\in X$, and $f_s^{-1}(0,1]$ is contained in an element of $\mathcal{U}$ for each $s\in S$. As shown in \cite{Dyd} it is convenient to aggregate $\{f_s\}_{s\in S}$
into one function $f:X\to l_1(S)$ which becomes continuous. And that object is called in this paper a partition of unity instead of the traditional $\{f_s\}_{s\in S}$.

\begin{Definition}
$l_1(V)$ is the set of functions $\alpha:V\to \mathbb{R}$ satisfying $\sum\limits_{v\in V}|\alpha(v)| < \infty$.\\
The subset $\{v\in V | \alpha(v)\ne 0\}$ is called the \textbf{carrier} (or \textbf{support} of $\alpha$).
Notice it is always countable.
\end{Definition}

Each $v\in V$ has its \textbf{Kronecker delta function} $\delta_v:V\to \mathbb{R}$ which we will quite often identify with $v$.

\begin{Definition}
For each $v\in V$ there is a projection $\pi_v:l_1(V)\to \mathbb{R}$ defined by $\pi_v(\alpha)=\alpha(v)$ (it is the restriction of the evaluation function $\mathbb{R}^V\to \mathbb{R}$). By the \textbf{open star} $\st(v)$ of $v\in V$ we mean
$\pi_v^{-1}(\mathbb{R}\setminus \{0\})$. Thus $\{\st(v)\}_{v\in V}$ forms an open cover of
non-zero vectors in $l_1(V)$.
\end{Definition}

Given a non-zero function $f:X\to l_1(V)$ on a metric space $X$ our general strategy is
to measure it both by its Lipschitz number and by the numerical aspects of the cover
$\{f^{-1}(\st(v))\}_{v\in V}$ of $X$ (mostly its diameter and its Lebesgue number \ref{LebesgueDef}).

\begin{Definition}
Suppose $f:X\to l_1(V)$ is a non-zero function on a metric space $X$ and $M > 0$.
$f$ is called $M$-\textbf{cobounded} if $\diam(f^{-1}(\st(v))) < M$ for each $v\in V$.\\
$f$ is called \textbf{cobounded} if there is $M > 0$ such that $f$ is $M$-cobounded.
\end{Definition}

\begin{Definition}
Suppose $f:X\to l_1(V)$ is a non-zero function. % on a metric space $X$.
The \textbf{Lebesgue number} $\Leb(f)$ of $f$ is defined as the Lebesgue number
of $\{f^{-1}(\st(v))\}_{v\in V}$ (see \ref{LebesgueDef}).
\end{Definition}

\begin{Definition}\label{PartitionsOfUDef}
A \textbf{partition of unity} on $X$ is a function $f:X\to l_1(V)$ such that the $l_1$-norm of each $f(x)$, $x\in X$, is $1$ and $f(x)(v)\ge 0$ for all $v\in V$.\\
A partition of unity is called \textbf{simplicial} if the carrier of each $f(x)\in l_1(V)$ is finite.
$f$ is called \textbf{$n$-dimensional} if the carrier of each $f(x)\in l_1(V)$ contains at most $n+1$ points for each $x\in X$.
\end{Definition}

The easiest way to create a partition of unity on a set $X$ is to define a non-negative function $f:X\to l_1(V)$ and then to \textbf{normalize} it ($x\mapsto \frac{f(x)}{\|f(x)\|}$).

\begin{Definition}\label{LambdaCLipschitzDef}
A function $f:X\to Y$ of metric spaces is \textbf{$(\lambda,C)$-Lipschitz} if
$$d_Y(f(x),f(y))\leq \lambda\cdot d_X(x,y)+C$$
for all $x,y\in X$.
\end{Definition}

Cencelj-Dydak-Vavpeti\v c \cite{CenDyVav1} realized that the proper dualization of continuity in the case of \ref{MainThmOnPara} is the concept of a function being $(\lambda,C)$-Lipschitz and defined large scale paracompact spaces.

\begin{Definition}\label{LSParaDef} \cite{CenDyVav1}
$X$ is \textbf{large scale paracompact} if for each $\mu > 0$
there is a simplicial partition of unity $f:X\to l_1(V)$ (see \ref{PartitionsOfUDef}) satisfying the following conditions:\\
a. $f$ is $(\mu,\mu)$-Lipschitz,\\
b. the cover of $X$ induced by $f$ (the carriers of $f$) is uniformly bounded and is a coarsening of the cover of $X$ by $\frac{1}{\mu}$-balls.
\end{Definition}

The earlier definition of exact spaces by Dadarlat-Guentner is weaker.
\begin{Definition}\label{ExactSpacesDef} \cite{DaGu}
%A metric space
$X$ is \textbf{exact} if for each $r,\epsilon > 0$
there is a partition of unity $f:X\to l_1(V)$ satisfying the following conditions:\\
a. $f$ has $(r,\epsilon)$-variation (that means $d_Y(f(x),f(y))< \epsilon$ if $d_X(x,y) < r$),\\
b. the cover of $X$ induced by $f$ (the carriers of $f$) is uniformly bounded.
\end{Definition}

Conditions a) in both definitions are equivalent:\\
1. Given $(r,\epsilon)$, one only needs $\mu\cdot r+\mu <\epsilon$ to see a) of \ref{LSParaDef} implies a) of \ref{ExactSpacesDef}.\\
2. Given $2 > \mu > 0$, one only needs
$\epsilon=\mu$ and $r > \frac{2-\mu}{\mu}$ to see a) of \ref{ExactSpacesDef} implies a) of \ref{LSParaDef}. It has to do with the diameter of the unit sphere of $l_1(V)$ being $2$.

The missing ingredient in \ref{ExactSpacesDef} is the thickness of the cover of $X$ induced by $f$. The same problem is with the original definition of the Property A of G.~Yu. Both work well for the class of spaces of bounded geometry (that means for each $r$ there is an upper bound on the number of points of $B(x,r)$ for all $x\in X$). However, for general metric spaces one needs to make adjustments in order for the theory to work.

We need the concept of a \textbf{contraction} of a partition of unity.

\begin{Definition}\label{ContractionOfPU}
If $f:X\to l_1(V)$ is a partition of unity and $\alpha:V\to S$ is a surjection,
then by the \textbf{contraction} of $f$ along $\alpha$ we mean $\alpha_{\ast}\circ f:X\to l_1(S)$, where $\alpha_{\ast}:l_1(V)\to l_1(S)$
is the induced linear map.
\end{Definition}

\begin{Lemma}\label{LipschitzOfContraction}
Suppose $g$ is a contraction of a partition of unity $f:X\to l_1(V)$.\\
a. $\Leb(g)\ge \Leb(f)$.\\
b. If $f$ is $(\epsilon,\epsilon)$-Lipschitz for some $\epsilon > 0$, then $g$ is $(\epsilon,\epsilon)$-Lipschitz.
\end{Lemma}

\begin{proof}
a. The covering of $X$ induced by $g$ is a coarsening of the cover induced by $f$. Therefore $\Leb(g)\ge \Leb(f)$.

b. Notice $\alpha_{\ast}$ has the norm at most $1$ (it is so in view of the Triangle Inequality), hence it is $(1,0)$-Lipschitz which implies $\alpha_{\ast}\circ f:X\to l_1(S)$ is $(\epsilon,\epsilon)$-Lipschitz.
\end{proof}

\begin{Proposition}\label{NewPartition}
Suppose $f:X\to l_1(V)$ is a non-zero $M$-cobounded function % on a metric space $X$
for some $M > 0$.
If $f^{-1}(\st(v)) \ne \emptyset$ for each $v\in V$, then there is an injection $\alpha: V\to X\times \mathbb{N}$ so that the composition $g:X\to l_1(X\times \mathbb{N})$ of $f$ and the induced linear map
$\alpha_{\ast}:l_1(V)\to l_1(X\times \mathbb{N})$ has the property that $g^{-1}(\st(x,n))\subset B(x,M)$
for all $(x,n)\in X\times \mathbb{N}$.
\end{Proposition}

\begin{proof}
For each $x\in X$ enumerate all vertices $w$ satisfying $f(x)(w)\ne 0$
as $v(x,1)$, $v(x,2)$, $\ldots$ For each $w\in V$ pick $x(w)\in f^{-1}(\st(w))$
and then pick the unique $n\in \mathbb{N}$ so that $w=v(x(w),n)$. Now set
$\alpha(w)=(x(w),n)$.

Since $v(\alpha(w))=w$, $\alpha$ is injective.

Suppose $y\in g^{-1}(\st(x,n))$. Put $w=v(x,n)$. Therefore $f(y)(w)\ne 0$ and
$f(x)(w)\ne 0$ resulting in $x,y\in f^{-1}(\st(w))$ which implies $d(x,y) < M$.
Thus $y\in B(x,M)$.
\end{proof}

\section{Barycentric partitions of unity}\label{BarycentricPUs}

In order to unify all the concepts via partitions of unity we created the notion of a \textbf{barycentric partition of unity} and we use it to explain and generalize Property A. This is part of our general strategy to explain most concepts via partitions of unity (see \cite{Dyd} for an exposition of basic topology from the point of view of partitions of unity).

\begin{Definition}\label{StandardDefOfPropA}
A \textbf{barycentric partition of unity} is $f:X\to l_1(V)$ such that $f(x)$ is of the form
$\frac{\chi_{C(x)}}{|C(x)|}$ for each $x\in X$.\\
Thus $f$ is the normalization of $F$ such that each $F(x)$ is the characteristic function (or the indicator function) of a finite subset $C(x)$ of $V$.
\end{Definition}

As each barycentric partition of unity is simplicial, the cover of $X$ induced by them
is point-finite.

\begin{Definition}
If $\mathcal{U}=\{U_s\}_{s\in S}$ is a point-finite cover of $X$ then its \textbf{induced
barycentric partition of unity} $p_{\mathcal{U}}:X\to l_1(S)$ is the normalization
of $f(x)=\sum \{\delta_s | x\in U_s\}$.
\end{Definition}

Thus there is a one-to-one function from point-finite covers of $X$ to barycentric partitions of unity on $X$. Observe, however, that if $f:X\to l_1(V)$ is a barycentric partition of unity on $X$ and $\mathcal{U}$ is the cover of $X$ by point-inverses of open stars $\st(v)$, $v\in V$,
then $p_{\mathcal{U}}$ may differ from $f$. Indeed, one may have
$f^{-1}(\st(v))=f^{-1}(\st(w))$ and $v\ne w$.

\begin{Lemma}\label{BasicLemmaOnBaryPUs}
For every two non-empty finite subsets $A$ and $B$ of $S$ one has
$$ \frac{|A\triangle B|}{\max(|A|,|B|)}\leq  \frac{|A\setminus B|}{|A|}+\frac{|B\setminus A|}{|B|} \leq \left\| \frac{\chi_{A}}{|A|} - \frac{\chi_{B}}{|B|}\right\|
\leq 2\cdot \frac{|A\triangle B|}{\min(|A|,|B|)}$$
in $l_1(S)$.
\end{Lemma}

\begin{proof} $A\Delta B:= (A\setminus B)\cup (B\setminus A)$ is the \textbf{symmetric difference} of $A$ and $B$. Notice that
$$\| |A|\cdot \chi_B - |B|\cdot \chi_A\|=
|A|\cdot |A\setminus B|+|B|\cdot |B\setminus A|+|A\cap B|\cdot ||A|-|B||.$$
Divide both sides by $|A|\cdot |B|$ and perform easy estimations.
\end{proof}

\subsection{Amenability and barycentric partitions of unity} Let us show how amenability of a group can be easily introduced using barycentric partitions of unity.\\
One can introduce large scale geometry on a group $G$ by declaring uniformly bounded families to be exactly those refining $\{g\cdot F\}_{g\in G}$ for some finite subset $F\subset G$
of $G$ (see Brodskiy-Dydak-Mitra \cite{BDM}).
That structure is metrizable if and only if $G$ is countable and, in case of finitely generated groups, is identical with the coarse structure induced by a word metric on $G$.\\
It is natural to consider barycentric partitions of unity on $G$ of the form
$$\phi_F(x) = \frac{\chi_{x\cdot F}}{|F|}.$$

Recall that a \textbf{F\o lner sequence} for a group $G$ is a sequence of finite subsets $F(1)\subset F(2)\subset\ldots$ of $G$ such that $\bigcup\limits_{n=1}^\infty F(n)=G$ and $\lim\limits_{n\to\infty}\frac{|g F(n)\Delta F(n)|}{|F(n)|}=0$ for all $g\in G$.

\begin{Proposition}
Let $F(1)\subset F(2)\subset\ldots$ be a sequence of finite subsets of a group $G$ such that for all $n$ the barycentric partition of unity $\phi_{F(n)}$ is $(\epsilon_n,\epsilon_n)$-Lipschitz but not $(\epsilon,\epsilon)$-Lipschitz for $\epsilon<\epsilon_n$.
Then the following conditions are equivalent:\\
a. $\lim\limits_{n\to\infty} \epsilon_n = 0$,\\
b. $\{F(n)\}_{n\ge 1}$ is a F\o lner sequence.
\end{Proposition}

\begin{proof} Notice $\frac{|xF\triangle yF|}{|F|}=\frac{|x^{-1}yF\triangle F|}{|F|}$ for each $x,y\in G$ and each finite subset $F$ of $G$. Lemma~\ref{BasicLemmaOnBaryPUs} says
$$ \frac{|x^{-1}yF(n)\triangle F(n)|}{|F(n)|} = \| \phi_{F(n)}(x)-\phi_{F(n)}(y)\|_1
\leq 2\cdot \frac{|x^{-1}yF(n)\triangle F(n)|}{|F(n)|}$$

That means a) is equivalent to
$$\lim_{n\to\infty}\frac{|gF(n)\triangle F(n)|}{|F(n)|} =0$$
for every $g\in G$. That is the defining condition for a F\o lner sequence.
\end{proof}

\subsection{Barycentric partitions of unity and Property A}

Let us recall the original definition of Property A of G.Yu (see \cite{Yu00}  or \cite{YuNo}).

\begin{Definition}
A metric space $X$ has \textbf{Property A} if
for each $R > 0$ and each $\epsilon > 0$ there is $S > 0$ and a function $A$ from $X$ to finite subsets of $X\times \mathbb{N}$ satisfying the following properties:\\
a. $A(x)\subset B(x,S)\times \mathbb{N}$ for each $x\in X$,\\
b. if $d(x,y) < R$, then $A(x)\cap A(y)\ne \emptyset$ and
$$\frac{|A(x)\Delta A(y)|}{|A(x)\cap A(y)|} < \epsilon.$$
\end{Definition}

Let us show that Property A of Yu can be defined by replacing arbitrary partitions of unity in \ref{ExactSpacesDef} by barycentric partitions of unity.

\begin{Proposition}\label{CharPropAOfYuViaBaryPUs}
A metric space $X$ has, for every $\epsilon > 0$,
an $(\epsilon,\epsilon)$-Lipschitz barycentric partition of unity on $X$ that is cobounded, if and only if it has Property A.
\end{Proposition}
\begin{proof}
($\Rightarrow$)
Let $\epsilon, R>0$. Let $\bar\epsilon=\min\{\epsilon, \tfrac 1 2\}\tfrac1{R+1}$. There exists barycentric $(\bar\epsilon,\bar\epsilon)$-Lipschitz $M$-cobounded partition of unity $f: X\to l_1(V)$. By Proposition~\ref{NewPartition} there exists injection $\alpha: V\to X\times \mathbb{N}$ such that $g^{-1}(\st(x,n))\subset B(x,M)$ for all $(x,n)\in X\times \mathbb{N}$, where $g=\alpha_*\circ f$. Let $S=\alpha(V)$. Then $g:X\to l_1(S)$ is contraction of $f$ along $\alpha$ and it is also $M$-bounded barycentric partition of unity. By Lemma~\ref{LipschitzOfContraction} $g$ is $(\bar\epsilon,\bar\epsilon)$-Lipschitz. Let $A(x)=\{(y,n)\in X\times \mathbb{N}\mid g(x)(y,n)\ne 0\}$. Because $g$ is barycentric, $A(x)$ is finite for all $x$. Because $g^{-1}(\st(x,n))\subset B(x,M)$ for all $(x,n)\in X\times N$, $A(x)\subset B(x,M)\times \mathbb{N}$ for all $x\in X$.

Let $d(x,y)<R$. If $|A(x)\cap A(y)|<\tfrac 1 2 |A(x)|$, then
$$
\frac 1 2< \frac{|A(x)-A(y)|}{|A(x)|}\le \| g(x)-g(y)\|\le \bar\epsilon d(x,y)+\bar\epsilon<\bar\epsilon(R+1)<\tfrac 1 2
$$
which is a contradiction. Therefore $|A(x)|<2|A(x)\cap A(y)|$ for $d(x,y)<R$ in particular $A(x)\cap A(y)\ne \emptyset$.
By Lemma~\ref{BasicLemmaOnBaryPUs}
\begin{align*}
\frac{|A(x)\Delta A(y)|}{|A(x)\cap A(y)|}&\le \frac{|A(x)\Delta A(y)|}{2\max\{|A(x)|,|A(y)|\}}
\le \left\| \frac{\chi_{A(x)}}{|A(x)|} - \frac{\chi_{A(y)}}{|A(y)|}\right\|=\\
&=\left\| g(x) - g(y)\right\|\le \bar\epsilon d(x,y)+\bar\epsilon<\bar\epsilon (R+1)\le \epsilon.
\end{align*}

($\Leftarrow$)
Let $\epsilon>0$. By assumption there is a function $A$ from $X$ to finite subsets of $X\times \mathbb{N}$ such that $A(x)\subset B(x,S)\times \mathbb{N}$ for each $x\in X$ for some $S>0$ and for $d(x,y)<\tfrac{2-\epsilon}\epsilon$ the intersection $A(x)\cap A(y)\ne \emptyset$ and $\frac{|A(x)\Delta A(y)|}{|A(x)\cap A(y)|} < \tfrac\epsilon 2$. Then $f:X\to l_1(X\times \mathbb{N})$ defined as $f(x)=\tfrac{\chi_{A(x)}}{|A(x)|}$ is $S$-cobounded barycentric partition of unity. If $d(x,y)\ge \tfrac{2-\epsilon}\epsilon$ then
$$
\|f(x)-f(y)\|\le 2=\epsilon \tfrac{2-\epsilon}\epsilon+\epsilon\le \epsilon d(x,y)+\epsilon.
$$
If $d(x,y)< \tfrac{2-\epsilon}\epsilon$ then by Lemma~\ref{BasicLemmaOnBaryPUs}
$$
\|f(x)-f(y)\|\le 2 \frac{|A(x)\Delta A(y)|}{\min\{|A(x)|, |A(y)|\}}\le 2\frac{|A(x)\Delta A(y)|}{|A(x)\cap A(y)|}<2\frac\epsilon 2\le \epsilon d(x,y)+\epsilon,
$$
therefore $f$ is $(\epsilon,\epsilon)$-Lipschitz.
\end{proof}

\subsection{Creation of barycentric partitions of unity}

\begin{Definition}\label{ExpansionOfPU}
If $f:X\to l_1(V)$ is a partition of unity,
then by an \textbf{expansion} of $f$ we mean any partition of unity $g$ so that $f$ is its contraction.
\end{Definition}

\begin{Proposition}\label{ExpansionToBaryPU}
Suppose $f:X\to l_1(V)$ is a cobounded partition of unity that is $(\epsilon,\epsilon)$-Lipschitz for some $\epsilon > 0$.
If $f$ is the normalization of an integer-valued function $F:X\to l_1(V)$ (that means $F(x)(v)\in \mathbb{Z}_+$ for all $(x,v)\in X\times V$) such that the norm function 
$x\to \| F(x)\| $ is constant, then
there is a barycentric expansion $g$ of $f$ that is cobounded, $(\epsilon,\epsilon)$-Lipschitz, and $\Leb(g)=\Leb(f)$.
\end{Proposition}

\begin{proof}
Let $S=\{(v,n)\in V\times \mathbb{N} | F(x)(v)\le n \mbox{ for some }x\in X\}$, and let $\alpha:S\to V$ be the projection onto the first coordinate.
Define $G:X\to l_1(S)$ by $G(x)(v,i)=1$ if $F(x)(v)\leq i$ and $G(x)(v,i)=0$ if $F(x)(v) > i$.
Then $\sum_{i\in \mathbb{N}} G(x)(v,i)=F(x)(v)$ for all $x\in X$ and $v\in V$, therefore
$\|G(x)\|=\sum_{(v,i)\in S} G(x)(v,i)=\sum_{v\in V}F(x)(v)=\|F(x)\|$.
Let $g$ be the normalization of $G$.
Then $f$ is the contraction of $g$ along $\alpha$ and
\begin{align*}
\| f(x)-f(y)\|&=\sum_{v\in V}\left|\frac{F(x)(v)}{\|F(x)\|}-\frac{F(y)(v)}{\|F(y)\|}\right|=\\
&=\sum_{v\in V}\left|\frac{\sum_j G(x)(v,j)}{\|G(x)\|}-\frac{\sum_j G(y)(v,j)}{\|G(y)\|}\right|=\\
&=\sum_{v\in V}\sum_j\left|\frac{G(x)(v,j)}{\|G(x)\|}-\frac{G(y)(v,j)}{\|G(y)\|}\right|=\\
&=\| g(x)-g(y)\|
\end{align*}
for all $x,y\in X$. Hence $g$ is $(\epsilon,\epsilon)$-Lipschitz.
Because $f^{-1}(\st(v))=g^{-1}(\st(v,1))\supset g^{-1}(\st(v,n))$ for every $n\in \mathbb{N}$, $g$ is cobounded and $\Leb(g)=\Leb(f)$.
\end{proof}

\begin{Proposition}\label{SeparableAreLSWeakPara}
If $X$ is separable at scale $r$ (that means there is a countable subset $S$
of $X$ with $B(S,r)=X$), then there is a $6r$-cobounded barycentric partition of unity $f$ on $X$ whose Lebesgue number is at least $r$.
\end{Proposition}

\begin{proof} (due to Misha Levin) Enumerate elements of $S$ as $x_1$, $x_2$, $\ldots$
Put $U_n=B(x_n,2r)$, $V_n=U_n\setminus \bigcup\limits_{i=1}^{n-1} U_i$, and
$W_n=B(V_n,r)$. Notice $\{V_n\}$ is a cover of $X$, so the Lebesgue number of $\mathcal{W}=\{W_n\}$
is at least $r$.\\
Given $x\in X$ choose $m\ge 1$ so that $d(x,x_m) < r$. Notice $B(x,r)\subset B(x_m,2r)$, so
$B(x,r)\cap V_n=\emptyset$ for all $n > m$. Hence $x\notin W_n$ for all $n > m$.\\
Let $f=p_{\mathcal{W}}$ and it is clear $f$ is a $6r$-cobounded barycentric partition of unity on $X$ whose Lebesgue number is at least $r$.
\end{proof}

\section{Large scale paracompactness}\label{LSParacompactness}

In this section we investigate properties of large scale weakly paracompact spaces and their interaction with large scale paracompact spaces.

\begin{Exercise}
A topological space $X$ is weakly paracompact if and only if for each open cover $\mathcal{U}$ of $X$ there is a barycentric partition of unity $f$ on $X$ so that the cover of $X$ induced by $f$ is open and refines $\mathcal{U}$.
\end{Exercise}

\begin{Proposition}\label{LSWeakParaIsAnInvariant}
If $X$ coarsely embeds in a large scale weakly paracompact space $Y$, then $X$ is 
large scale weakly paracompact.
\end{Proposition}
\begin{proof}
Suppose $f:X\to Y$ is a coarse embedding. Given $r,s > 0$ find $r', s' > 0$ with the following properties:\\
a. $d_X(x,y) < r$ implies $d_Y(f(x),f(y)) < r'$,\\
b. $d_Y(f(x),f(y)) < s'$ implies $d_X(x,y) < s$.\\
Pick a uniformly bounded cover $\mathcal{U}$ of $Y$ of Lebesgue number at least $s'$ such that every $r'$-ball $B(z,r')$ is contained in only finitely many elements of $\mathcal{U}$. Define $\mathcal{V}$ as $f^{-1}(\mathcal{U})$ and observe
$\mathcal{V}$ is of Lebesgue number at least $s$ such that every $r$-ball $B(x,r)$ is contained in only finitely many elements of $\mathcal{V}$.
\end{proof}

\begin{Proposition}\label{CharacterisationLSWP}
The following conditions are equivalent for each metric space $X$:\\
a. For each $r > 0$ there is a uniformly bounded cover $\mathcal{U}$ of $X$ such that every $r$-ball $B(x,r)$ intersects only finitely many elements of $\mathcal{U}$.\\
b. $X$ is large scale weakly paracompact.\\
c. For every uniformly bounded cover $\mathcal U$ of $X$ there exists uniformly bounded point-finite cover $\mathcal V$ such that $\mathcal U$ is refinement of $\mathcal V$.\\
d. For each $M > 0$ there exists a cobounded barycentric partition of unity $f\colon X\to l_1(V)$ of Lebesgue number at least $M$.\\
e. For each $M > 0$ there exists a cobounded simplicial partition of unity $f\colon X\to l_1(V)$ of Lebesgue number at least $M$.
\end{Proposition}
\begin{proof}
a)$\implies$b). Suppose $r,s > 0$.
Pick a cover $\mathcal{V}$ such that every $(r+s)$-ball $B(x,r+s)$ intersects only finitely many elements of $\mathcal{V}$. Notice $B(x,r)\subset B(A,s)$ implies $B(x,r+s)\cap A\ne\emptyset$ for any subset $A$ of $X$. Therefore, the family $\mathcal{U}:=\{B(V,s) | V\in \mathcal{U}\}$ is a uniformly bounded cover of $X$ of Lebesgue number at least $s$ such that every $r$-ball $B(x,r)$ is contained in only finitely many elements of $\mathcal{U}$. According to Definition
\ref{LSWeakParaDef},
$X$ is large scale weakly paracompact.\\
b)$\implies$c). Suppose $\mathcal U$ is a uniformly bounded cover of $X$. Put $r=\diam(\mathcal{U})+1$, $s=2r$, and find a uniformly bounded cover $\mathcal W$ of $X$ of Lebesgue number at least $s$ such that every $r$-ball $B(x,r)$ is contained in only finitely many elements of $\mathcal{W}$. Given $A\subset X$ define $B(A,-r)$ as $X\setminus B(X\setminus A,r)$ and observe $x\in B(A,-r)\implies B(x,r)\subset A$.
Therefore, the family $\mathcal{V}:=\{B(W,-r) | W\in \mathcal{W}\}$ is a uniformly bounded cover of $X$ of Lebesgue number at least $r$ such that every $x\in X$ is contained in only finitely many elements of $\mathcal{V}$. Also, $\mathcal{V}$ coarsens $\mathcal{U}$.\\
c)$\implies$d). Given $M>0$ there is a point-finite uniformly bounded cover $\mathcal V$ such that $\{B(x,M)\mid x\in X\}$ is a refinement of $\mathcal V$. The standard partition of unity $p_{\mathcal V}$ is cobounded simplicial and of Lebesgue number at least $M$.\\
d)$\implies$e) is obvious.\\
e)$\implies$a).
Let $r > 0$. There exists a cobounded simplicial partition of unity $f\colon X\to l_1(S)$ of Lebesgue number at least $r+1$. Consider $\mathcal{U}=\{B(\st (s),-r)\mid s\in S\}$.
It is a uniformly bounded cover of $X$ such that every $r$-ball $B(x,r)$ intersects only finitely many elements of $\mathcal{U}$ as $B(x,r)\cap B(A,-r)\ne \emptyset\implies x\in A$.
\end{proof}

\begin{Corollary}
If $X$ is large scale separable, then it is large scale weakly paracompact.
\end{Corollary}
\begin{proof}
$X$ is large scale separable if there is a countable set $S$ of $X$ such that $X=B(S,r)$ for some $r > 0$. Let $M>0$ and $\bar M=\max\{M,r\}$. Then $B(S,\bar M)=X$ and by Proposition~\ref{SeparableAreLSWeakPara} there exists a cobounded partition of unity on
$X$ whose Lebesgue number is at least $\bar M\ge M$. By Proposition~\ref{CharacterisationLSWP} $X$ is large scale weakly paracompact.
\end{proof}

\begin{Problem}
Is every metric space large scale weakly paracompact?
\end{Problem}

Use \ref{CharacterisationLSWP} to prove the following.
\begin{Corollary}
Every large scale paracompact space $X$ is large scale weakly paracompact.
\end{Corollary}

We do not know if we can weaken Definition \ref{LSParaDef} by dropping the assumption of partitions of unity being simplicial.

\begin{Problem}\label{AltLSParaDef}
Let $X$ be a metric space such that for each $\epsilon > 0$
there is a partition of unity $f:X\to l_1(V)$ satisfying the following conditions:\\
a. $f$ is $(\epsilon,\epsilon)$-Lipschitz,\\
b. the cover of $X$ induced by $f$ (the carriers of $f$) is uniformly bounded and is a coarsening of the cover of $X$ by $\frac{1}{\epsilon}$-balls.\\
Is $X$ large scale paracompact?
\end{Problem}

We will show the answer to \ref{AltLSParaDef} is positive if $X$ is large scale weakly paracompact.
\begin{Lemma}\label{PUReductionToSimplicial}
Suppose $1 > \epsilon > 0$. If $f:X\to l_1(V)$ is an $(\frac{\epsilon}{2},\frac{\epsilon}{2})$-Lipschitz partition of unity on $X$ that is cobounded, then there is a simplicial partition of unity $g:X\to l_1(V)$ that is $(\epsilon,\epsilon)$-Lipschitz and is cobounded.
\end{Lemma}
\begin{proof}
For each $x\in X$ pick a finite subset $C(x)$ of the carrier of $f(x)$ such that
$$\sum\limits_{v\notin C(x)}f(x)(v) < \frac{\epsilon}{4}.$$
Define $g(x)$ by setting $g(x)(v)=0$ for all $v\notin C(x)$, then picking $v(x)\in C(x)$
and setting $g(x)(v(x))=f(x)(v(x))+\sum\limits_{v\notin C(x)}f(x)(v)$. For $v\in C(x)\setminus \{v(x)\}$ we put $g(x)(v)=f(x)(v)$.
\end{proof}

We are ready to show that the difference between exact spaces of Dadarlat-Guentner \cite{DaGu} (see \ref{ExactSpacesDef}) and large scale paracompact spaces is large scale weak paracompactness.
\begin{Theorem}\label{MainThmOnLSParacompactness}
If $X$ is large scale weakly paracompact
and for each $\epsilon > 0$
there is an $(\epsilon,\epsilon)$-Lipschitz partition of unity on $X$ that is cobounded,
then $X$ is large scale paracompact.
\end{Theorem}

\begin{proof}
Given $\epsilon > 0$ pick a cover $\{U_s\}_{s\in S}$ of $X$ consisting of non-empty sets that is $M$-cobounded and every ball $B(x,\frac{1}{\epsilon})$ intersects only finitely many elements of $\{U_s\}_{s\in S}$. For each $s\in S$ pick $x_s\in U_s$.

For each $x\in X$ let $S(x)=\{s\in S | B(x,\frac{1}{\epsilon})\cap U_s\ne \emptyset\}$.

Pick $\delta < \frac{\epsilon}{2M+1}$ and pick a simplicial partition of unity $f:X\to l_1(V)$ on $X$
that is cobounded and $(\delta,\delta)$-Lipschitz using \ref{PUReductionToSimplicial}.

Define a new partition of unity $g$ on $X$ by the formula
$$g(x)=\frac{\sum\limits_{s\in S(x)} f(x_s)}{|S(x)|}.$$
Notice it is cobounded.

Given $x\in X$ choose $s\in S$ so that $x\in U_s$ and choose $v\in V$
satisfying $f(x_s)(v)\ne 0$. If
$y\in B(x,\frac{1}{\epsilon})$, then $s\in S(y)$ so $g(y)(v)\ne 0$ and $y\in g^{-1}(\st(v))$.
That proves the Lebesgue number of $g$ is at least $\frac{1}{\epsilon}$.

Given $x,y\in X$,
$$|S(x)|\cdot |S(y)|\cdot (g(x)-g(y))=\sum\limits_{s\in S(x)} |S(y)|\cdot f(x_s)
-\sum\limits_{t\in S(y)} |S(x)|\cdot f(x_t)$$
can be rewritten as the sum of $|S(x)|\cdot |S(y)|$ differences of the form
$$f(x_s)-f(x_t)$$
where $s\in S(x)$ and $t\in S(y)$.
Therefore $d(x_s,x_t) < 2M+d(x,y)$ implying $\|f(x_s)-f(x_t)\|\leq \delta (2M+d(x,y))+\delta$.
Thus
$$|S(x)|\cdot |S(y)|\cdot \|g(x)-g(y)\|\leq |S(x)|\cdot |S(y)|\cdot (\delta (2M+d(x,y))+\delta)$$
resulting in
$$\|g(x)-g(y)\|\leq \delta (2M+d(x,y))+\delta < \epsilon\cdot d(x,y)+\epsilon$$
as we can assume $M>1/2$.
\end{proof}

The meaning of Theorem \ref{MainThmOnLSParacompactness} is that, in the class of large scale weakly paracompact spaces, exact spaces of Dadarlat-Guentner coincide with the class of large scale paracompact spaces.

\begin{Corollary}\label{LSParaIsAnInvariant}
If $X$ coarsely embeds in a large scale paracompact space $Y$, then $X$ is 
large scale paracompact.
\begin{proof}
By \ref{LSWeakParaIsAnInvariant} $X$ is large scale weakly paracompact.
Suppose $f:X\to Y$ is a coarse embedding. Pick a sequence $f_n:X\to l_1(V_n)$ of cobounded partitions of unity that are $(\frac{1}{n},\frac{1}{n})$-Lipschitz and observe that
$g_n=f_n\circ f$ is a sequence of cobounded partitions of unity such that for some sequence $\epsilon_n\to 0$, $g_n$ is $(\epsilon_n,\epsilon_n)$-Lipschitz.
\end{proof}
\end{Corollary}

\section{Strong Property A}\label{StrongPropertyASection}

According to Rules \ref{RulesOfDualizationRemark} the Property A has another variant similar to large scale paracompactness. In this section we call it Strong Property A and we show that, in the class of large scale finitistic spaces, large scale paracompactness and strong Property A are equivalent. Large scale finitistic spaces contain all spaces of bounded geometry. As the concept of bounded geometry is not a coarse invariant, we introduce a new coarse invariant (called coarsely doubling) that encompasses all spaces of bounded geometry.

\begin{Problem}\label{PropertyAQuestion}
Is $X$ large scale paracompact if it has, for every $\epsilon > 0$,
an $(\epsilon,\epsilon)$-Lipschitz barycentric partition of unity on $X$ that is cobounded?
\end{Problem}

\begin{Remark}
In view of \ref{MainThmOnLSParacompactness} it suffices to show $X$ is large scale weakly paracompact in order to answer \ref{PropertyAQuestion} in the positive.
\end{Remark}

We want to strengthen Yu's \cite{Yu00} definition of Property A to arbitrary metric spaces so that spaces with strong Property A are large scale paracompact.

\begin{Definition}\label{NewStrongPropADef}
A metric space $X$ has \textbf{strong Property A} if for every $\epsilon > 0$
there is an $(\epsilon,\epsilon)$-Lipschitz barycentric partition of unity on $X$ that is cobounded and whose Lebesgue number is at least $\frac{1}{\epsilon}$.
\end{Definition}

\begin{Observation}
As in \ref{LSParaIsAnInvariant} one can show that if
 $X$ coarsely embeds in a space $Y$ with strong Property A, then $X$ has 
strong Property A.
\end{Observation}

\begin{Definition}\label{LSFinitisticDef}
A metric space $X$ is \textbf{large scale finitistic} if for every $r > 0$
there is a uniformly bounded cover $\mathcal{U}$ of $X$ whose Lebesgue number
$\Leb(\mathcal{U})$ is at least $r$ and there is $n(\mathcal{U})=n > 0$ such that each $x\in X$ belongs to at most $n$ elements of $\mathcal{U}$.
\end{Definition}

Recall the concept of doubling spaces from analysis (\cite{Hei}, p.81).

\begin{Definition}\label{DoublingDef}
A metric space $X$ is \textbf{doubling} if for every $r > 0$ there is a natural number $n(r)$ such that every $2r$-ball can be covered by at most $n(r)$ many $r$-balls.
\end{Definition}

Here is a natural generalization of doubling spaces.

\begin{Definition}\label{CoarselyDoublingDef}
A metric space $X$ is \textbf{coarsely doubling} (or \textbf{large scale doubling}) if there is $M > 0$ with the property that for every $r > M$ there is a natural number $n(r)$ such that every $2r$-ball can be covered by at most $n(r)$ many $r$-balls.
\end{Definition}

\begin{Proposition}\label{CoarselyDoublingProperties}
a. Every space of bounded geometry is doubling.\\
b. Every coarsely doubling space $X$ is large scale finitistic.\\
c. Every coarsely doubling space $X$ contains a subspace $Y$ of bounded geometry
such that the inclusion $Y\to X$ is a coarse equivalence.
\end{Proposition}
\begin{proof}
a. Obviously, every space of bounded geometry is doubling.\\
b) and c). Suppose there is $M > 0$ such that for every $r > M$ there is a natural number $n(r)$ such that every $2r$-ball can be covered by at most $n(r)$ many $r$-balls.\\
Assume $r > 2M$.
Choose a maximal subset $Y=\{x_s\}_{s\in S}$ of $X$ with the property that $d(x_s,x_t)\ge r$ for each $s\ne t$ in $S$. Given $x\in X$ consider $T=\{s\in S | x_s\in B(x,2r)\}$.
Notice $|T|\leq n(\frac{r}{2})\cdot n(r)$ as otherwise $B(x,2r)$ cannot be covered by a set of $\frac{r}{2}$-balls containing at most
$n(\frac{r}{2})\cdot n(r)$ elements (that would result in two elements $x_s,x_t$, $s,t\in T$, to end up in the same element of the cover). That means the horizon
of $x$ in $\{B(x_s,2r)\}_{s\in S}$ contains at most $n(\frac{r}{2})\cdot n(r)$ elements and $X$ is large scale finitistic due to $\Leb(\{B(x_s,2r)\}_{s\in S})\ge r$.\\
Use $Y$ as above for $r=2M+1$. Put $r(m)=2^{m-1}\cdot r$ for $m\ge 1$.
Notice that $B(x,r(m+1))\cap Y$ contains at most $n(\frac{r}{2})\cdot n(r)\cdot\ldots \cdot n(r(m))$ points for all $m\ge 1$. For $m=1$ it has been just proved. For general $m$ it follows by induction.
\end{proof}

\begin{Proposition}\label{CoarselyDoublingChar}
$X$ is coarsely doubling if and only if there is $R > 0$ and a function $N:(0,\infty)\to \mathbb{N}$ such that any $s$-ball can be covered by at most $N(s)$ balls of radius $R$.
\end{Proposition}
\begin{proof}
The condition in the statement of the proposition is clearly stronger than $X$ being coarsely doubling. Indeed, put $M=R$ and $n(r)=N(2r)$.

Suppose there is $M > 0$ with the property that for every $r > M$ there is a natural number $n(r)$ such that every $2r$-ball can be covered by at most $n(r)$ set of $r$-balls. Put $R=M+1$ and notice that every ball of radius $s=2^k\cdot R$, $k\ge 2$, can be covered by at most $n(R)\cdot n(2R)\cdot \ldots \cdot n(2^{k-1}R)$ balls of radius $R$.
\end{proof}

\begin{Corollary}\label{CoarselyDoublingSubspace}
A subspace of a coarsely doubling space is coarsely doubling.
\end{Corollary}
\begin{proof}
Suppose there is $R > 0$ and a function $N:(0,\infty)\to \mathbb{N}$ such that any $s$-ball in $X$ can be covered by at most $N(s)$ balls of radius $R$. If $Y\subset X$ notice every $s$-ball in $Y$ can be covered by at most $N(s)$ balls of radius $2R$. Indeed, given a set of at most $N(s)$ balls $B_i$ in $X$ of radius $R$ covering a ball centered in $Y$ of radius $s$, the corresponding $s$-ball in $Y$ is covered by balls in $Y$ of radius $2R$ centered at points in $B_i\cap Y$ provided that intersection is non-empty (otherwise that ball is discarded).
\end{proof}

\begin{Corollary}\label{CoarselyDoublingInvariant}
Being coarsely doubling is a coarse invariant.
\end{Corollary}
\begin{proof}
It suffices to consider a coarse equivalence $f:X\to Y$ that is a surjection and $Y$ is coarsely doubling. Suppose there is $R > 0$ and a function $N:(0,\infty)\to \mathbb{N}$ such that any $s$-ball in $Y$ can be covered by at most $N(s)$ balls of radius $R$. 

Pick $D > 0$ such that $d_Y(f(x),f(y)) < R$ implies $d_X(x,y) < D$. 
Pick a function $\alpha:(0,\infty)\to (0,\infty)$ with the property that $d_Y(f(x),f(y)) < \alpha(d_X(x,y))$. Given any ball $B$ in $X$ of radius $s$, its image $f(B)$ is contained in an $\alpha(s)$-ball and can be covered by at most $N(\alpha(s))$ balls $B_i$ in $Y$ of radius $R$. Since each $f^{-1}(B_i)$ is contained in a $D$-ball in $X$, $B$ can be covered by at most $N(\alpha(s))$ balls in $X$ of radius $D$.
\end{proof}

\begin{Corollary}\label{CoarselyDoublingLattice}
The class of coarsely doubling spaces is identical with the class of spaces that have a quasi-lattice.
\end{Corollary}
\begin{proof}
A \textbf{quasi-lattice} (see \cite{BW}) is a subset $\Gamma$ of $X$ that is $c$-dense for some $c > 0$ and $\Gamma$ is of bounded geometry. By \ref{CoarselyDoublingProperties} every coarsely doubling space has a quasi-lattice.

Assume $X$ has a quasi-lattice $\Gamma$.
By \ref{CoarselyDoublingProperties} $\Gamma$ is coarsely doubling, and by
\ref{CoarselyDoublingInvariant} $X$ is coarsely doubling.
\end{proof}

The following theorem generalizes known results on Property A for spaces of bounded geometry (see \cite{Willett}, \cite{Guen}) and spaces of finite asymptotic dimension (see \cite{CenDyVav1} and \cite{CenDyVav2}).

\begin{Theorem}\label{PropAForFinitistic}
A large scale finitistic metric space $X$ has strong Property A if and only if it is large scale paracompact.
\end{Theorem}
\begin{proof} As can be seen from \ref{NewStrongPropADef} strong Property A implies large scale paracompactness, so only one direction in the proof is of interest.

Suppose $X$ is large scale finitistic and large scale paracompact. Given $\epsilon > 0$
we choose an $(\epsilon,\epsilon)$-Lipschitz simplicial partition of unity $f:X\to l_1(V)$ that is
cobounded and $\Leb(f)\ge \frac{1}{\epsilon}$.\\
Pick a uniformly bounded cover $\{U_s\}_{s\in S}$ of multiplicity at most $n+1$ which is a coarsening of the cover of $X$ induced by $f$. Let $\alpha:V\to S$ be defined so that
$$f^{-1}(\st(v))\subset U_{\alpha(v)}$$
for each $v\in V$. We may assume $\alpha$ is surjective by removing elements $U_t$ of the cover $\{U_s\}_{s\in S}$ such that $t\in S\setminus \alpha(V)$.

Consider the contraction $g=\alpha_{\ast}\circ f:X\to l_1(S)$ of $f$. Notice it is $(\epsilon,\epsilon)$-Lipschitz by \ref{LipschitzOfContraction}, cobounded, $\Leb(g)\ge \frac{1}{\epsilon}$, and $g$ is $n$-dimensional (in view of $g^{-1}(\st(s))\subset U_s$ for each $s\in S$).

Find a natural number $m\ge \frac{2(n+1)}{\epsilon}+(n+1)\cdot (n+2)$. Consider $G=m\cdot g$ and
express it as the sum $G_1+G_2$, where $G_2\ge 0$ is integer-valued, $G_2(x)(v) > 0$
iff $G(x)(v) > 0$, $\| G_2(x)\|=m$ for each $x\in X$, and $\| G_1(x)\|\leq 2n+2$ for each $x\in X$. The way to do it is to set initially $G_1(x)(v)$ to be equal to $G(x)(v)-1$ if $0 < G(x)(v) < 1$, $G_1(x)(v)=0$ if $G(x)(v)=0$, and $G_1(x)(v)=G(x)(v) -\lfloor G(x)(v)\rfloor$
if $G(x)(v) \ge 1$ (here $\lfloor x\rfloor$ is the integer part of $x$).

Let $k(x)=\sum\limits_{v\in V} G_1(x)(v)$ and $G_2=G-G_1$. Notice $k(x)$ is an integer-valued function and $|k(x)|$< n+1.
For every $x\in X$ proceed as follows. If $k(x)<0$, then there is $w\in V$ with $G_2(x)(w)>|k(x)|$,
in which case we assign $G_1(x)(w)-k(x)$ as the new $G_1(x)(w)$ and we assign $G_2(x)(w)+k(x)$ as
the new $G_2(x)(w)$. If $k(x)\geq 0$, then we pick any $w\in V$ such that $G(x)(w)>0$ and we
assign $G_1(x)(w)-k(x)$ as the new $G_1(x)(w)$ and we assign $G_2(x)(w)+k(x)$ as the new $G_2(x)(w)$.

Let $h:X\to l_1(S)$ be the normalization of $G_2$. Notice $|h(x)-g(x)|\leq \frac{2n+2}{m} < \epsilon$ for each $x\in X$, so $h$ is $(2\epsilon,2\epsilon)$-Lipschitz. Also, it induces the same cover of $X$ as $g$ does which implies $h$ is cobounded and $\Leb(h)\ge \frac{1}{\epsilon}$.
The function $h$ can be expanded to a barycentric partition of unity $p$ by \ref{ExpansionToBaryPU} that is $(2\epsilon,2\epsilon)$-Lipschitz, is cobounded and $\Leb(p)\ge \frac{1}{\epsilon}$.
\end{proof}

\section{Coarse amenability}\label{CoarseAmenabilitySection}

In this section we investigate properties of coarsely amenable spaces. The main result of this section is that (on the class of coarsely doubling spaces) coarse amenability, Property A, and large scale paracompactness are all equivalent.

\begin{Observation}
As in \ref{LSWeakParaIsAnInvariant} one can show that if
 $X$ coarsely embeds in a coarsely amenable space $Y$, then $X$ is coarsely amenable.
\end{Observation}

\begin{Proposition}\label{CharOfStrongPropA}
The following conditions are equivalent for every metric space $X$;\\
a. $X$ is coarsely amenable.\\
b. For each $r > 0$ and each $\epsilon > 0$ there is a uniformly bounded cover
$\mathcal{U}$ of $X$ such that for each $x\in X$ the horizon $\hor(B(x,2r),\mathcal{U})$ is finite and
$$\frac{|\hor(B(x,r),\mathcal{U})|}{|\hor( B(x,2r),\mathcal{U})|} > 1- \epsilon.$$
c. For each $r > 0$ and each $\epsilon > 0$ there is a uniformly bounded cover
$\mathcal{U}$ of $X$ such that for each $x\in X$ the horizon $\hor(B(x,r),\mathcal{U})$ is finite and
$$\frac{|\hor(x,\mathcal{U})|}{|\hor( B(x,r),\mathcal{U})|} > 1- \epsilon.$$
d. For each $s > r > 0$ and each $M,\epsilon > 0$ there is a uniformly bounded cover
$\mathcal{U}$ of $X$ of Lebesgue number at least $M$ such that for each $x\in X$ the horizon $\hor(B(x,s),\mathcal{U})$ is finite and
$$\frac{|\hor(B(x,r),\mathcal{U})|}{|\hor( B(x,s),\mathcal{U})|} > 1- \epsilon.$$
\end{Proposition}
\begin{proof} a)$\implies$b) and d)$\implies$a) are obvious.\\
b)$\implies$c). Given $r > 0$ and $\epsilon > 0$ pick
a uniformly bounded cover
$\mathcal{V}=\{V_t\}_{t\in T}$ of $X$ such that for each $x\in X$ the horizon $\hor(B(x,2r),\mathcal{V})$ is finite and
$$\frac{|\hor(B(x,r),\mathcal{V})|}{|\hor( B(x,2r),\mathcal{V})|} > 1- \epsilon.$$
Define $U_t=B(V_t,r)$ and put $\mathcal{U}=\{U_t\}_{t\in T}$. Notice $B(x,r)\cap B(V_t,r)\ne \emptyset$ implies $B(x,2r)\cap V_t\ne\emptyset$. That means
$$\hor(B(x,r),\mathcal{U})\subset \hor(B(x,2r),\mathcal{V}).$$
Similarly, $\hor(x,\mathcal{U})=\hor(B(x,r),\mathcal{V})$. Therefore
$$\frac{|\hor(x,\mathcal{U})|}{|\hor( B(x,r),\mathcal{U})|} > 1- \epsilon.$$
c)$\implies$d). Suppose $s > r > 0$ and $M,\epsilon > 0$. Pick a uniformly bounded cover
$\mathcal{V}=\{V_t\}_{t\in T}$ of $X$ such that for each $x\in X$ the horizon $\hor(B(x,s+M),\mathcal{V})$ is finite and
$$\frac{|\hor(x,\mathcal{V})|}{|\hor( B(x,s+M),\mathcal{V})|} > 1- \epsilon.$$
Define $U_t=B(V_t,M)$ and put $\mathcal{U}=\{U_t\}_{t\in T}$. Notice $B(x,s)\cap B(V_t,M)\ne \emptyset$ implies $B(x,s+M)\cap V_t\ne\emptyset$. That means
$$\hor(B(x,s),\mathcal{U})\subset \hor(B(x,s+M),\mathcal{V}).$$
Since $\hor(x,\mathcal{V})\subset \hor(B(x,r),\mathcal{U})$,
$$\frac{|\hor(B(x,r),\mathcal{U})|}{|\hor( B(x,s),\mathcal{U})|} > 1- \epsilon.$$

\end{proof}

\begin{Proposition}\label{StrongPropAImpliesPropA}
Every coarsely amenable space $X$ has strong Property A.
\end{Proposition}
\begin{proof}
Given $\epsilon > 0$ consider $r,\mu > 0$ to be determined later and pick $\delta > 0$ so that
$$1+\mu > \frac{1}{1-\delta}.$$
Then pick
a uniformly bounded cover $\mathcal{U}$ of $X$ such that $\Leb(\mathcal{U})\ge 4r$ (see \ref{CharOfStrongPropA}) and for each $x\in X$ the horizon $\hor(B(x,2r),\mathcal{U})$ is finite and
$$\frac{|\hor(B(x,r),\mathcal{U})|}{|\hor( B(x,2r),\mathcal{U})|} > 1- \delta.$$

For each $x\in X$ let $A(x)=\hor( B(x,2r),\mathcal{U})$ and $D(x)=\hor(B(x,r),\mathcal{U})$.\\
Define the barycentric partition of unity $g:X\to l_1(S)$ ($S$ being the set indexing $\mathcal{U}$) as the normalization
of the function $f(x)=\chi_{A(x)}$. Notice that $\Leb(g)\ge 2r$. \\
If $d(x,y) < r$, then $D(x)\subset A(x)\cap A(y)$. Thus
$$|A(x)| < (1+\mu)\cdot |D(x)|\leq (1+\mu)\cdot |A(x)\cap A(y)|$$
resulting in
$$|A(x)\setminus A(y)| < \mu \cdot |A(x)\cap A(y)|.$$
Using \ref{BasicLemmaOnBaryPUs} we get that $d(x,y) < r$ implies
$\| g(x)-g(y)\| < 4\mu$.\\
If we request $r > \frac{1}{\epsilon}$, we get $\Leb(g) > \frac{1}{\epsilon}$.
If we request $\mu < \frac{\epsilon}{4}$ and $r > \frac{2-\epsilon}{\epsilon}$, then we get
$g$ is $(\epsilon,\epsilon)$-Lipschitz. Indeed, in case $d(x,y)\ge r$ it is automatic
($\epsilon \cdot d(x,y)+\epsilon > 2\ge \| g(x)-g(y)\| $ in this case), and
$d(x,y) < r$ implies $\| g(x)-g(y)\| < 4\mu \leq \epsilon$.
\end{proof}

\begin{Theorem}
For a coarsely doubling metric space $X$ the following conditions are equivalent:\\
a. $X$ is large scale paracompact,\\
b. $X$ has Property A,\\
c. $X$ is coarsely amenable.
\end{Theorem}
\begin{proof}
a)$\equiv$b) follows from \ref{PropAForFinitistic} and \ref{MainThmOnLSParacompactness}. In view of \ref{StrongPropAImpliesPropA} it suffices to show b)$\implies$c).

Using \ref{CoarselyDoublingProperties} we can reduce it to $X$ of bounded geometry.
Suppose $s > 0$. Pick $M > 0$ so that each $s$-ball $B(x,s)$, $x\in X$, contains at most $M$ points.

Given any $\mu > 0$ find a uniformly bounded cover
$\mathcal{U}(\mu)$ such that the barycentric partition of unity $p_{\mathcal{U}(\mu)}$ induced by $\mathcal{U}(\mu)$ is $(\mu,\mu)$-Lipschitz.

Given $x\in X$ let $A(x)=\hor(x,\mathcal{U}(\mu))$.
By \ref{BasicLemmaOnBaryPUs} it implies
$$|A(x)\Delta A(y)| < (s+1)\cdot\mu\cdot \max(|A(x)|,|A(y)|)$$
whenever $d(x,y) < s$. Therefore
$$|A(y)| < \frac{|A(x)|}{1-(s+1)\cdot\mu}$$
whenever $d(x,y) < s$.

Enumerate all points $y\in B(x,s)$ as $y_1,\ldots, y_k$ for some $k\leq M$. Now
\begin{align*}
|\bigcup\limits_{i=1}^k A(y_i)|&\leq |A(x)\cup \bigcup\limits_{i=1}^k (A(y_i)\setminus A(x))|\leq |A(x)|+ \sum\limits_{i=1}^k |A(x)\Delta A(y_i))|\leq \\
&\le |A(x)|+\frac{M\cdot (s+1)\cdot\mu\cdot |A(x)|}{1-(s+1)\cdot\mu}=
(1+\frac{M\cdot (s+1)\cdot\mu}{1-(s+1)\cdot\mu})\cdot |A(x)|
\end{align*}
Given any $\epsilon > 0$ we may choose $\mu > 0$ so that $(1+\frac{M\cdot (s+1)\cdot\mu}{1-(s+1)\cdot\mu})^{-1} > 1- \epsilon$.
Notice $\bigcup\limits_{i=1}^k A(y_i)=\hor( B(x,s),\mathcal{U}(\mu))$.
Since $A(x)=\hor(x,\mathcal{U}(\mu))$,
$$\frac{|\hor(x,\mathcal{U}(\mu))|}{|\hor( B(x,s),\mathcal{U}(\mu))|} > 1- \epsilon.$$
Thus $X$ is coarsely amenable by \ref{CharOfStrongPropA}.
\end{proof}

\begin{Corollary}\label{HilbertIsNotLSPara}
The Hilbert space is not large scale paracompact.
\end{Corollary}
\begin{proof}
As shown in \cite{AGS}, the Hilbert space contains a bounded geometry subspace (the box space of the free group of two generators) that does not have Property A. Hence the Hilbert space cannot be large scale paracompact.
\end{proof}
\begin{Remark}
One cannot derive \ref{HilbertIsNotLSPara} from earlier result of P.Nowak \cite{nowak}
(who constructed a subspace of the Hilbert cube without Property A)
as his subspace is not of bounded geometry.
\end{Remark}

\begin{Problem}\label{SimpleProofHProblem}
Find a direct/simple proof of the Hilbert space not being large scale paracompact.
\end{Problem}

\section{Expanders and coarse amenability}\label{ExpandersSection}

The purpose of this section is to prove that three major classes are not coarsely amenable: expanders, graph sequences with girth approaching infinity, and spaces constructed by P.Nowak \cite{nowak}.

Let $G$ be an undirected graph with vertex set $V(G)$ and edge set $E(G)$. For a collection of vertices $A \subseteq V(G)$, let $\partial A$ denote the collection of all edges going from a vertex in $A$ to a vertex outside of $A$:
$$ \partial A := \{ (x, y) \in E | x \in A, y \in V(G) \setminus A \}.$$
(Remember that edges are unordered, so the edge $(x, y)$ is the same as the edge $(y, x)$.)
\begin{Definition}\label{CheegerConstantDef}
The \textbf{Cheeger constant} of a finite graph $G$, denoted $h(G)$, is defined by
$$    h(G) := \min \left\{ \left. \frac{| \partial A |}{| A |} \right| A \subseteq V(G), 0 < | A | \leq \frac{| V(G) |}{2} \right\}.$$
\end{Definition}

The Cheeger constant is strictly positive if and only if $G$ is a connected graph. Intuitively, if the Cheeger constant is small but positive, then there exists a "bottleneck", in the sense that there are two "large" sets of vertices with "few" links (edges) between them. The Cheeger constant is "large" if any possible division of the vertex set into two subsets has "many" links between those two subsets (see Wikipedia).
%\url{http://en.wikipedia.org/wiki/Cheeger_constant_%28graph_theory%29}

\begin{Definition}\label{ExpanderDef}
A finite graph $G$ is a \textbf{$(k,\varepsilon)$-expander} if each vertex of $G$ has valency (degree) at most $k$, and $h(G)\geq \varepsilon > 0$.

A sequence of finite connected graphs $\{G_i\}$ is called a \textbf{graph sequence} if $|G_i|\rightarrow \infty$.

A graph sequence $\{G_i\}$ is called an \textbf{expander sequence} if there exists $k,\varepsilon$ such that each $G_i$ is a $(k,\varepsilon)$-expander.
\end{Definition}

Expander sequences were defined by Bassalygo and Pinsker in 1973 \cite{BP}. It is not obvious that such sequences exist. Their existence was first proved by Pinsker \cite{Pin}, in a non-constructive way. Margulis was the first to give explicit examples of expanders using discrete groups with property (T) \cite{Mar73} , \cite{Mar88}. For more on expanders see \cite{HLW}.

From the point of view of large scale geometry expander sequences are too restrictive and we will weaken their definition.

\begin{Definition}\label{HaloDef}
Given a subset $A$ of a finite graph $G$ its \textbf{halo} (denoted by $\halo(A)$) is the set of points not in $A$ such that their $2$-ball (in the graph metric) intersects $A$.
\end{Definition}

\begin{Definition} [\cite{YuNo}, p.18 for finite spaces]
Given a sequence $\{(X_n,d_n)\}$ of bounded metric spaces their \textbf{coarse disjoint union}
$X=\bigcup\limits_{n=1}^\infty X_n$ is the disjoint union of all $X_n$ equipped with the metric
$d:X\times X\to [0,\infty)$ defined as follows:\\
a. $d|X_n\times X_n=d_n$ for each $n$.\\
b. $d(x,y)=\diam(X_i)+\diam(X_j)$ if $x\in X_i$, $y\in X_j$ and $i\ne j$.

\end{Definition}

\begin{Definition}\label{ExpanderLightDef}
A graph sequence $\{G_i\}$ is an \textbf{expander light sequence} if there is $c > 0$
such that for any uniformly bounded sequence of subsets $A_n\subset G_n$
one has $|\halo(A_n)|\ge c\cdot |A_n|$ for infinitely many $n$.

The size of $A_n$ is measured using the metric from the coarse disjoint union 
$ \bigcup\limits_{i=1}^\infty G_i$, where each $G_i$ is equipped with the graph metric.
\end{Definition}

\begin{Remark}
Notice that H.Sako \cite{Sato} (see Definition 2.4) independently came up with a concept equivalent to our expander light sequences under the name of \lq a sequence of weak expander spaces\rq\ and in a more general context of coarse spaces.
\end{Remark}

\begin{Proposition}\label{AltExpanderDef}
Any expander sequence $\{G_i\}$ is an expander light sequence.
\end{Proposition}
\begin{proof} Suppose each $G_i$ is a $(k,\varepsilon)$-expander.
Given a uniformly bounded sequence $\{A_n\}$ of subsets of $G_n$ there is $M > 0$
such that $A_n$ contains at most half of vertices of $G_n$ for $n > M$.

The collection $\partial A_n$ of all edges going from a vertex in $A_n$ to a vertex outside of $A_n$ has
at least $\epsilon\cdot |A_n|$ elements. Their endpoints not in $A_n$ form exactly the set $C=\halo(A_n)$ of points not in $A_n$ such that their $2$-ball intersects $A_n$. Since each point $c\in C$ can produce at most $k$ edges in $\partial A_n$, $|C|\ge \frac{|\partial A_n|}{k}\ge \epsilon \frac{|A_n|}{k}$, and $c=\frac{\epsilon}{k}$ works.
\end{proof}

\begin{Theorem}\label{ExpanderLightThm}
Let $X =\bigcup\limits_{n=1}^\infty G_n$ be the coarse disjoint union of a graph sequence.
If $\{G_n\}$ is an expander light sequence,
then $X$ is not coarsely amenable.
\end{Theorem}

\begin{proof}
Pick a uniformly bounded cover
$\mathcal{U}=\{U_s\}_{s\in S}$ of $X$ such that for each $x$
$$\frac{|\hor(B(x,1),\mathcal{U})|}{|\hor( B(x,2),\mathcal{U})|} > p > \frac{1}{1+c}.$$
Restrict the cover to the graph $G=X_m$, with $m$ chosen so that elements $U_s$ of $\mathcal{U}$ intersecting $X_m$ must be contained in $X_m$ and
the halo of $U_s$ contains at least $c\cdot |U_s|$ elements. All sufficiently large $m$
have the property that elements of $\mathcal{U}$ intersecting $X_m$ must be contained in $X_m$ (as they are uniformly bounded and, for large $m$, $X_m$ is far away from its complement in the graph sequence) and among those, infinitely many $m$ have the property that the halo of $U_s$ contains at least $c\cdot |U_s|$ elements.\\
Let $P$ be the set of pairs $(x,s)$ such that $x\notin U_s$ but $B(x,2)$ intersects $U_s$. 
By fixing $s$ and counting points $x\in U_s$ such that $(x,s)\in P$, we see that
$$|P|\ge c\cdot \sum\limits_{s\in S} |U_s|.$$
Also,
$$|P|\leq \frac{1-p}{p}\cdot  \sum\limits_{s\in S} |U_s|.$$
Indeed,
\begin{align*}
|P|&=\sum\limits_{x\in G} (|\hor( B(x,2),\mathcal{U})|-|\hor( B(x,1),\mathcal{U})|) <\\
&<\frac{1-p}{p} \sum\limits_{x\in G} |\hor( B(x,1),\mathcal{U})|=\frac{1-p}{p}\cdot  \sum\limits_{s\in S} |U_s|.
\end{align*}
Therefore
$$ c \leq \frac{1-p}{p}$$
and there is a bound on $p$ from above
$$p \leq \frac{1}{1+c}.$$
\end{proof}

\begin{Corollary}
The coarse disjoint union of any expander sequence does not have Property A.
\end{Corollary}

\begin{Remark}
See \cite{KhWr} for a proof of the fact that expander sequences do not have Property A using cohomology. See \cite{BNW} for a cohomology characterization of Property A.
\end{Remark}

\begin{Proposition}
Suppose $\{G_i\}$ is a graph sequence with all vertices of degree at least three. If $girth(G_n)\to\infty$, then 
$\{G_i\}$ is an expander light sequence.
\end{Proposition}
\begin{proof}
By \textbf{girth} (notation: $girth(G)$) of a graph $G$ we mean the minimum number of edges
of a non-trivial loop in $G$.

Given $M > 0$ consider $k$ so that $girth(X_n) > 4M$ for $n \ge k$. Notice any $2M$-ball in $X_n$ is a tree and for any subset $A$ of $X_n$ containing at most $M$ elements, the number of points in the halo of $A$ is at least $|A|$. This can be shown by induction. It is clearly so if $A$ contains only one point or is empty. For arbitrary $A$, we pick $a_0\in A$ and denote by $B$ the set of points in $A$ at the maximal distance from $a_0$. Put $ A_1=A\setminus B$.
Notice each point of $B$ contributes at least $2$ points to $\halo(A)$ that are not in $\halo(A_1)$. Therefore $|\halo(A)|\ge |\halo(A_1)\setminus B|+2\cdot |B|\ge |A_1|-|B|+2|B|=|A_1|+|B|=|A|$.
\end{proof}

\begin{Corollary} [R.Willett \cite{Willett2}]
Let $X = \bigcup\limits_{n=1}^\infty X_n$ be the coarse disjoint union of a graph sequence 
with degrees of all vertices at least three. Assume that $girth(X_n)\to\infty$. Then $X$ does
not have Property A if it is of bounded geometry (that amounts to an upper bound on the degrees of all vertices).
\end{Corollary}

\begin{Proposition}\label{ProductOfGroupsProp}
Suppose $\{(G_i,\Sigma_i)\}$ is an infinite sequence of finite groups and their generators. For each $n$ let $S_n$ be the set of generators of $H_n=G_1\times\ldots\times G_n$ obtained as the union of sets $\{1\}\times\ldots\times \{1\}\times \Sigma_i\times \{1\}\times\ldots\times \{1\}$. If all $G_i$ are non-trivial, then the sequence of Cayley graphs of $(H_i,S_i)$ is an expander light sequence.
\end{Proposition}
\begin{proof}

\textbf{Claim}: If $A\subset H_n$ contains the unit $1$, $M$ is a natural number, $A\subset B(1,M)$, and $n > 3M+2$,
then $|\halo(A)|\ge |A|$.\\
\textbf{Proof of Claim}: It is clearly so for $M=1$. Suppose $M=k+1$ and $k\ge 1$.
Let $C$ be the set of points in $A$ at distance $k$ from $1$ and let $A_1=A\setminus B$. By induction, $\halo(A_1)$ contains at least $|A_1|$ elements.

Notice that $\halo(B)$ contains at least $\frac{(n-M)\cdot |B|}{M+1}\ge 2|B|$ elements that do not belong to $A\cup \halo(A_1)$. Indeed, given an element $b$ of $B$ one can put a generator of $G_j$
in the $j$th slot of $b$ if it is equal to $1$. That results in possible counting of each element $M+1$ times which explains the denominator.

Now $|\halo(A)|\ge |\halo(A_1)-B|+2|B|\ge |A|$.

Since each subset $A$ of $H_n$ is isometric to a subset containing $1$, we can drop the assumption of $1\in A$ in the Claim and derive \ref{ProductOfGroupsProp} that way.
\end{proof}

We are now able to deduce in a simple way that the Hilbert space is not coarsely amenable (see \ref{SimpleProofHProblem}).
\begin{Corollary}\label{InfiniteProductOfGroups}
Consider a finite group $G\ne 1$ with a metric $d_G$. If each $G_n=G\times \ldots \times G$ is equipped with the $l_1$ metric, then the coarse disjoint union $X=\bigcup G_n$ is not coarsely amenable. In particular, the Hilbert space is not coarsely amenable.
\end{Corollary}
\begin{proof}
Notice that choosing another metric $\rho$ on $G$ results in bi-Lipschitz equivalence $f$
of $(G,d_G)$ and $(G,\rho)$. $f$ can be extended to bi-Lipschitz equivalence from $(G_n,d_G)$ to $(G_n,\rho)$ with the same constants. Thus the coarse type of $X$
does not depend on $d_G$.

If we choose $d_G$ to be the word metric induced by a set of generators $\Sigma$
and consider the word metric on $G_n$ induced by $S_n$ as in \ref{ProductOfGroupsProp}, then that metric is exactly the $l_1$-metric induced by $d_G$. Thus $X$ is not coarsely amenable.

If we choose $G=\mathbb{Z}/2\mathbb{Z}$ and $d_G$ is induced by the inclusion $G\to \mathbb{R}$ ($0\to 0$ and $1\to 1$), then $X$ embeds coarsely into the Hilbert space.
\end{proof}

\begin{Remark}
Notice \cite{nowak} (see also \cite{YuNo}, pp.78--81) contains a stronger result than \ref{InfiniteProductOfGroups} (namely those spaces do not have Property A).
However, our proof is much simpler.

From the four known classes of spaces that do not have Property A only warped cones of J.Roe \cite{Roe2} are so far not related to \ref{ExpanderLightThm}.

\end{Remark}

\section{Addendum}\label{AddendumSection}

After the first version of this paper was written, J. Brodzki, G. Niblo, J. \v Spakula, R. Willett, and N. Wright \cite{BNSWW} came up with another way of showing that the three classes in Section \ref{ExpandersSection} do not have Property A using the concept of \textbf{uniform local amenability} ULA. In addition, they discussed the relatonship of ULA to MSP (\textbf{metric sparsification property}). This section is aimed at justifying the following diagram illustrating the relationship between analogs of amenability discussed in \cite{BNSWW} and coarse amenability:
$$
\xymatrix{
 \mbox{coarse amenability}\ar[r]& \mbox{MSP} &\iff & \mbox{ULA}_\mu
}
$$
Again, in case of spaces of bounded geometry, all three concepts are equivalent.

\begin{Definition}
Given $R > 0$ and given a subset $E$ of a metric space, by {\bf $R$-boundary} $\partial_R E$ of $E$ we mean the set of points in $X\setminus E$ of distance less than $R$ to $E$.
\end{Definition}

\begin{Definition}\label{ULAmuDef}
A metric space $X$ has property ULA$_\mu$ (\textbf{uniform local amenability with respect to probability measures}) if for all $R, \epsilon > 0$ there exists $S > 0$ such that for all probability measures $\lambda$ on $X$ there exists a finite subset $E$ of $X$ of diameter at most $S$ satisfying
$$\lambda(\partial_R E) < \epsilon\cdot \lambda(E).$$
\end{Definition}

\begin{Remark}
Notice that the $R$-boundary $\partial_R E$ of $E$ in \cite{BNSWW} is defined as the set of points in $X\setminus E$ of distance at most $R$ to $E$. That does not lead to a different concept of ULA$_\mu$ as our $R$-boundary is contained in the $R$-boundary of \cite{BNSWW} and our $2R$-boundary contains the $R$-boundary of \cite{BNSWW}.
\end{Remark}

\begin{Remark}
Uniform local amenability ULA of \cite{BNSWW} can be obtained by restricting the above definition to uniform probability measures on $X$ (measures where each point is either of measure $0$ or a constant measure $c > 0$). It is not known if ULA is equivalent to Property A for bounded geometry spaces. On the other hand, it is shown in \cite{BNSWW} (using results of \cite{Sato2}) that Property A is equivalent to ULA$_\mu$ for spaces of bounded geometry.
\end{Remark}

\begin{Definition}\label{MSPDef}
A metric space $X$ has MSP (\textbf{Metric Sparsification Property}) if for all $R > 0$ and for all positive $c < 1$ there exists $S > 0$ such that for all probability measures $\mu$ on $X$ there exists an $R$-disjoint family $\{\Omega_i\}_{i\ge 1}$ of subsets of $X$ of diameter at most $S$ satisfying
$$\sum\limits_{i=1}^\infty \mu(\Omega_i) > c.$$
\end{Definition}

\begin{Proposition}
The following conditions are equivalent for a metric space $X$:
\begin{itemize}
\item[a.] $X$ has property ULA$_\mu$,
\item[b.] for all $R, \epsilon > 0$ there exists $S > 0$ such that for all probability measures $\mu$ on $X$ there exists a subset $E$ of $X$ of diameter at most $S$ satisfying
$$\mu(\partial_R E) < \epsilon\cdot \mu(E),$$
\item[c.] for all $R, \epsilon > 0$ there exists $S > 0$ such that for all probability measures $\mu$ on $X$ with finite support there exists a subset $E$ of $X$ of diameter at most $S$ satisfying
$$\mu(\partial_R E) < \epsilon\cdot \mu(E).$$
\end{itemize}
\end{Proposition}
\begin{proof}
Both implications a)$\implies$b) and b)$\implies$c) are obvious.\\
c)$\implies$a). We are going to show that $X$ has MSP and, for the benefit of the reader, we will show that MSP implies ULA$_\mu$ - see \cite{BNSWW} for a proof of equivalence of MSP and ULA$_\mu$.

Suppose $R > 0$. Given $0 < c < 1$ choose $\epsilon > 0$ so that $\frac{1+c}{2(1+\epsilon)} > c$. Given a probability measure $\mu$ on $X$
find a finite subset $Y$ of $X$ of measure at least $\frac{1+c}{2}$ and consisting of points of non-zero measure. By creating a probability measure on $Y$ via rescaling of $\mu$ we obtain a subset $Z$ of $X$ of diameter at most $S$
with the property that $\mu(Y\cap \partial_R Z) < \epsilon \cdot \mu(Y\cap Z)$.
Let $Z_1=Z\cap Y$. Notice $Z_1\ne\emptyset$ and $\mu(Y\cap \partial_R Z_1) < \epsilon \cdot \mu(Z_1)$. Applying the same procedure to $Y_1=Y\setminus B(Z_1,R)$
we find a subset $Z_2$ of $Y_1$ of diameter at most $S$ such that $\mu(Y\cap \partial_R Z_2) < \epsilon \cdot \mu(Z_2)$ unless $Y_1=\emptyset$. Proceeding by induction we construct a finite sequence of nonempty subsets $Z_i$, $1\leq i\leq n$, of $Y$ of diameter at most $S$ such that they are $R$-disjoint, $\mu(Y\cap \partial_R Z_i) < \epsilon \cdot \mu(Z_i)$ for each $i$, 
and $Y\subset \bigcup\limits_{i=1}^n B(Z_i,R)$.

Let $\Omega=\bigcup\limits_{i=1}^n Z_i$. Since $Y= \bigcup\limits_{i=1}^n(Z_i\cup (Y\cap\partial_R Z_i))$, $\frac{1+c}{2} < \mu(Y) < \sum\limits_{i=1}^n \mu(Z_i)+\epsilon\cdot \mu(Z_i)=(1+\epsilon)\mu(\Omega)$. Therefore $\mu(\Omega) > \frac{1+c}{2(1+\epsilon)} > c$. Thus $X$ has MSP.

Given $R, \epsilon > 0$ and given a probability measure $\mu$ on $X$ find a $2R$-disjoint family of finite sets $Z_i$, $1\leq i\leq n$, of diameter at most $S$ so that
$\mu(\bigcup\limits_{i=1}^n Z_i) > \max(1-\epsilon/2, 1/2)$. Since
$ 1 \ge \sum\limits_{i=1}^n (\mu(Z_i)+\mu(\partial_R Z_i))$, $\sum\limits_{i=1}^n \mu(\partial_R Z_i)\leq \epsilon/2 < \epsilon\cdot \sum\limits_{i=1}^n \mu(Z_i) $
and there is $j$ so that $\mu(\partial_R Z_j) < \epsilon\cdot \mu(Z_j)$.

\end{proof}

\begin{Theorem}
Every coarsely amenable space $X$ is uniformly locally amenable with respect to probability measures.
\end{Theorem}
\begin{proof}
Given $R, \epsilon > 0$, pick a uniformly bounded cover
$\mathcal{U}=\{U_s\}_{s\in S}$ of $X$ such that for each $x$
$$\frac{|\hor(x,\mathcal{U})|}{|\hor( B(x,R),\mathcal{U})|} > \frac{1}{1+\epsilon}.$$
Notice the above inequality is equivalent to
$$|\hor( B(x,R),\mathcal{U})|-|\hor(x,\mathcal{U})| <\epsilon\cdot |\hor(x,\mathcal{U})|.$$
Given a probability measure $\mu$ on $X$ of finite support extend it to the product measure
$\lambda$ on $X\times S$, where each point in $S$ is given the measure of $1$.
Let $P$ be the set of pairs $(x,s)$ such that $x\in \partial_R(U_s)$. Notice $\lambda(P)$ is finite as each point $x\in X$ has only finitely many $s\in S$ so that $x\in \partial_R(U_s)$.

By fixing $s$ and looking at points $x\in U_s$ such that $(x,s)\in P$, we see that
$$\lambda(P)= \sum\limits_{s\in S} \mu(\partial_R(U_s)).$$
Also,
$$\lambda(P) < \epsilon\cdot  \sum\limits_{s\in S} \mu(U_s).$$
Indeed, by fixing $x\in X$ and counting $s\in S$ so that $x\in \partial_R(U_s)$, we see that
\begin{align*}
\lambda(P)&=\sum\limits_{x\in X} \mu(x)\cdot (|\hor( B(x,R),\mathcal{U})|-|\hor(x,\mathcal{U})|) <\\
&< \epsilon \sum\limits_{x\in X} \mu(x)\cdot |\hor(x,\mathcal{U})|=\epsilon\cdot  \sum\limits_{s\in S} \mu(U_s).
\end{align*}
Therefore there is $t\in S$ so that 
$$ \mu(\partial_R(U_t)) < \epsilon\cdot \mu(U_t).$$
\end{proof}


\begin{thebibliography}{99}
\bibitem{ArensDug}
R. Arens, J. Dugundji, \emph{Remark on the concept of compactness}, Portugaliae Math. 9, (1950), 141--143.

\bibitem{AGS} G. N. Arzhantseva, E. Guentner, J. \v Spakula, \emph{Coarse non-amenability and coarse embeddings}, Geom. Funct. Anal. 22 (2012), 22--36

\bibitem{BP}
L. A. Bassalygo, M. S. Pinsker. \emph{The complexity of an optimal non-blocking commutation scheme without reorganization}. Problemy Pereda\v ci Informacii  9 (1973), no. 1, 84--87, 
Translated into English in Problems of Information Transmission 9 (1974),
64--66.

\bibitem{Bing}
R. H. Bing, \emph{Metrization of topological spaces}, Canadian J. Math. 3, (1951), 175--186.

\bibitem{BDM}
N. Brodskiy, J. Dydak, A. Mitra, \emph{Coarse structures and group actions},
Colloquium Mathematicum 111 (2008), 149--158.

\bibitem{BW}
J. Block and S. Weinberger, \emph{Aperiodic tilings, positive scalar curvature and amenability of spaces}, J. Amer. Math. Soc., 5(4):907--918, 1992.

\bibitem{BNW}
J. Brodzki, G. A. Niblo, N. Wright,
\emph{A cohomological characterisation of Yu's Property A for metric spaces}, Geometry and Topology, 16 (2010), 391--432.

\bibitem{BNSWW} J. Brodzki, G. A. Niblo, J. \v Spakula, R. Willett, and N. Wright, \emph{Uniform local amenability}, J. Noncommut. Geom. 7 (2013), no. 2, 583--603.

\bibitem{CenDyVav1}
M. Cencelj, J. Dydak, A. Vavpeti\v c, \emph{Asymptotic dimension, Property A, and Lipschitz maps}, Rev. Mat. Complut. 26 (2013), no. 2, 561--571. 

\bibitem{CenDyVav2}
M. Cencelj, J. Dydak, A. Vavpeti\v c, \emph{Property A and asymptotic dimension}, Glas. Mat. Ser. III 47(67) (2012), no. 2, 441--444.


\bibitem{DaGu}
M. Dadarlat, E. Guentner, \emph{Uniform embeddability of relatively hyperbolic groups}, J. Reine Angew. Math. 612 (2007), 1--15.

\bibitem{Dieu}
J. Dieudonn\'e, \emph{Une g\'en\'eralisation des espaces compacts}, J. Math. Pures Appl. (9) 23 (1944), 65--76.

\bibitem{Dyd}
J. Dydak, \emph{Partitions of unity}, Topology Proc. 27 (2003), no.1, 125--171.

\bibitem{Engelking}
R. Engelking, \emph{Theory of dimensions, finite and infinite}, Sigma Series in Pure Mathematics, 10. Heldermann Verlag, Lemgo, 1995.

\bibitem{Grom}
M. Gromov, \emph{Asymptotic invariants of infinite groups}, Geometric group theory, Vol. 2, 1--295, London Math. Soc. Lecture Note Ser., 182, Cambridge Univ. Press, Cambridge, 1993.

\bibitem{Guen}
E. Guentner, \emph{Permanence in coarse geometry}, Recent Progress in General Topology III, to appear.

\bibitem{Hei}
J. Heinonen, \emph{Lectures on Analysis on Metric Spaces}, Springer Verlag, Universitext 2001.

\bibitem{HLW}
S. Hoory, N. Linial, A. Wigderson, \emph{Expander graphs and their applications},
Bull. Amer. Math. Soc. (N.S.) 43 (2006), no. 4, 439–-561.

\bibitem{KhWr}
A. Khukhro, N. J. Wright, \emph{Expanders and property A}, Alg. Geom. Top. 12 (2012), 37--47.

\bibitem{Mar73}
G. A. Margulis, \emph{Explicit constructions of expanders}. Problemy Pereda\v ci Informacii,
9 (1973), no. 4, 71--80. 

\bibitem{Mar88}
G. A. Margulis, \emph{Explicit group-theoretic constructions of combinatorial schemes and their applications in the construction of expanders and concentrators}, Problems Inform. Transmission 24 (1988), no. 1, 39--46.

\bibitem{nowak}
P. W. Nowak,
\emph{Coarsely embeddable metric spaces without property {A}},
J. Funct. Anal. 252 (2007), no. 1, 126--136. 

\bibitem{YuNo}
P. W. Nowak, G. Yu, \emph{Large Scale Geometry}, EMS Textbooks in Mathematics, 
European Mathematical Society (EMS), Zürich, 2012.

\bibitem{Pin}
M. S. Pinsker, \emph{On the complexity of a concentrator}, In 7th International Telegraffic
Conference (1973), 318/1--318/4.

\bibitem{Roe lectures}
J. Roe, \emph{Lectures on coarse geometry}, University Lecture Series, 31. American Mathematical Society, Providence, RI, 2003.

\bibitem{Roe2}
J. Roe, \emph{Warped cones and property A}, Geom. Topol. 9 (2005), 163--178.

\bibitem{Sato}
H. Sako, \emph{A generalization of expander graphs and local reflexivity of uniform Roe algebras}, J. Funct. Anal. 265 (2013), no. 7, 1367--1391. 

\bibitem{Sato2}
H. Sako, \emph{Property A and the operator norm localization property of discrete metric spaces}, 
arXiv:1203.5496, 2012.

\bibitem{Willett}
R. Willett, \emph{Some notes on Property A}, Limits of graphs in group theory and computer science, 191--281, EPFL Press, Lausanne, 2009.

\bibitem{Willett2}
R. Willett, \emph{Property A and graphs with large girth}, J. Topol. Anal. 3 (2011), no. 3, 377--384.

\bibitem{Yu1998} G. Yu, \emph{The Novikov conjecture for groups with finite asymptotic dimension}, Ann. of Math. (2) 147 (1998), no. 2, 325--355.

\bibitem{Yu00}
G. Yu, \emph{The coarse Baum-Connes conjecture for spaces which admit a uniform embedding into Hilbert space}, Invent. Math. 139 (2000), no. 1, 201--240.

\end{thebibliography}
\end{document}